\documentclass[12pt]{amsart}
\usepackage{amscd}
\usepackage{amssymb}
\usepackage{amsfonts}
\usepackage[matrix,arrow,curve]{xy}

\newcommand{\bZ}{{\mathbb Z}}

\newcommand{\bP}{{\mathbb P}}
\newcommand{\bG}{{\mathbb G}}

\newcommand{\bF}{{\mathbb F}}
\newcommand{\Q}{\mathbb{Q}}
\newcommand{\Z}{\mathbb{Z}}

\newtheorem{thm}{Theorem}[section]
\newtheorem{lemma}[thm]{Lemma}

\newtheorem{prop}[thm]{Proposition}

\theoremstyle{definition}

\newtheorem{theosd}[thm]{Theorem}
\newtheorem{Propsd}[thm]{Proposition}
\newtheorem{rem}[thm]{Remark}

\numberwithin{equation}{section}

\begin{document}

\title{Note on  the counterexamples for the integral Tate conjecture
over finite fields}

\author{Alena Pirutka and Nobuaki Yagita}

\address{Alena Pirutka, Centre de Math\'ematiques Laurent Schwartz, UMR 7640 de CNRS, \'Ecole Polytechnique 91128 Palaiseau France}

\email{alena.pirutka@math.polytehcnique.fr}

\address{ Nobuaki Yagita, Department of Mathematics, Faculty of Education,
Ibaraki University,
Mito, Ibaraki, Japan}

\email{ yagita@mx.ibaraki.ac.jp}

\begin{abstract}
In this note we discuss some examples of non torsion and non algebraic cohomology classes for varieties over finite fields.
The approach follows the construction of Atiyah-Hirzebruch and Totaro.
\end{abstract}

\maketitle

\section{Introduction}

Let $k$ be a finite field and let $X$ be a smooth and projective variety over $k$. Denote $\bar k$ an algebraic closure of $k$ and $\mathfrak g=Gal(\bar k/k)$. Let $\ell$ be a prime, $\ell\neq char(k)$. The Tate conjecture \cite{Ta65} predicts that the cycle class map
$$CH^i(X_{\bar k})\otimes \Q_{\ell} \to \bigcup_U H^{2i}_{\acute{e}t}(X_{\bar k}, \Q_{\ell}(i))^U,$$
where the union is over all open subgroups $U$ of $\mathfrak g$,
is surjective.

In the integral version one is interested in the cokernel of the cycle class map
\begin{equation}\label{iTate} CH^i(X_{\bar k})\otimes \Z_{\ell} \to \bigcup_U H^{2i}_{\acute{e}t}(X_{\bar k}, \Z_{\ell}(i))^U.
 \end{equation}
 This map is not surjective in general: the counterexamples of Atiyah-Hirzebruch \cite{AH62}, revisited by Totaro \cite{To97},  to the integral version of the Hodge conjecture,  provide also counterexamples to the integral Tate conjecture \cite{CTSz10}. More precisely, one constructs an $\ell$-torsion class in $H^4_{\acute{e}t}(X_{\bar k}, \Z_{\ell}(2))$, which is not algebraic, for some smooth and projective variety $X$. However, one then wonders if there exists an example of a variety $X$ over a finite field, such that the map
\begin{equation}\label{cnt} CH^i(X_{\bar k})\otimes \Z_{\ell} \to \bigcup_U H^{2i}_{\acute{e}t}(X_{\bar k}, \Z_{\ell}(i))^U/torsion
 \end{equation}
 is not surjective  (\cite{Mi07, CTSz10}). In the context of an integral version of the Hodge conjecture, Koll\'ar \cite{Ko90} constructed such examples of curve classes. Over a finite field, Schoen \cite{Sch98} has proved that the map (\ref{cnt}) is always surjective for curve classes, if the Tate conjecture holds for divisors on surfaces.

  In this note we follow the approach of Atiyah-Hirzebruch and Totaro and we produce examples where the map (\ref{cnt}) is not surjective for $\ell=2,3\text{ or }5$. 

\theosd\label{ceProj}{\it Let $\ell$ be a prime  from the following list: $\ell=2,3\text{ or }5.$
There exists a smooth and projective variety $X$ over a finite field $k$,  $char k\neq \ell$, such that the cycle class map  $$CH^2(X_{\bar k})\otimes \Z_{\ell} \to \bigcup_U H^{4}_{\acute{e}t}(X_{\bar k}, \Z_{\ell}(2))^U/torsion$$ is not surjective.\\
}

As in the examples of Atiyah-Hirzebruch and Totaro, our counterexamples are obtained as a projective approximation of the cohomology of classifying spaces of some  simple simply connected groups, having $\ell$-torsion in its cohomology. The non algebraicity of a cohomology class is obtained by means of motivic cohomology operations: one establishes that  the operation $Q_1$ does not vanish on some class of degree $4$, but it always vanishes on the algebraic classes. This is done in section \ref{Operations}. Next, in section \ref{Cspaces} we discuss some properties of classifying spaces in our context and finally we construct a projective variety approximating the cohomology of these spaces in small degrees in section \ref{Proj}.\\

\paragraph{\bf Acknowledgements.}  This work has started during the Spring School and Workshop on
Torsors, Motives and Cohomological Invariants in the Fields Institute, Toronto, as a part of a Thematic Program on Torsors, Nonassociative Algebras and Cohomological Invariants (January-June 2013), organized by V. Chernousov, E. Neher, A. Merkurjev,
A. Pianzola and K. Zainoulline. We would like to thank the organizers and the Institute for their invitation, hospitality and support. We are very grateful to B. Totaro for his interest and for generously communicating his construction of a projective algebraic approximation in theorem \ref{ceProj}. The first author would like to thank B. Kahn and J. Lannes for useful discussions.

\section{Motivic version of Atiyah-Hirzebruch arguments, revisited}\label{Operations}
\subsection{Operations}
Let  $k$ be a perfect field with $char(k)\not =\ell$ and let $\mathcal H_{\cdot}(k)$ be the motivic homotopy theory of pointed $k$-spaces (see \cite{Mo-Vo}). For $X\in \mathcal H_{\cdot}(k)$, denote by $H^{*,*'}(X, \Z/\ell)$   the motivic cohomology groups with $\Z/\ell$-coefficients ({\it loc.cit.}). If $X$ is a smooth variety over $k$, note that one has an isomorphism  %$H^{*,*'}(X, \Z/\ell):=H^{*,*'}(X\sqcup \mathrm{Spec}k, \Z/\ell)$ and
$CH^*(X)/\ell\stackrel{\sim}{\to} H^{2*,*}(X, \Z/\ell)$.

Voevodsky \cite{Vo2} defined the reduced power operations $P^i$ and the Milnor's operations $Q_i$ on $H^{*,*'}(X, \Z/\ell)$:
 \[  P^i: H^{*,*'}(X,\bZ/{\ell})\to H^{*+2i(\ell-1),*'+i(\ell-1)}(X,\bZ/{\ell}), i\geq 0\]
 \[  Q_i: H^{*,*'}(X,\bZ/{\ell})\to H^{*+2\ell^{i}-1,*'+(\ell^i-1)}(X,\bZ/{\ell}), i\geq 0,\]
 where $Q_0=\beta$ is the Bockstein operation of degree $(1,0)$
induced from the short exact sequence
$0\to \bZ/\ell\stackrel{\times \ell}{\to} \bZ/\ell^2\to \bZ/\ell\to 0$
(see also \cite{Ri12}).

One of the key ingredients for this construction is the following computation of
 the motivic cohomology of
the classifying space $B\mu_{\ell}$ (\cite{Vo2}):
\begin{lemma}\label{mot-coh-Bmu}(\cite[\S 6]{Vo2})
For each object $X\in \mathcal H_{\cdot}(k)$, the graded algebra $
 H^{*,*'}(X\times B\mu_{\ell},\bZ/{\ell})$ is generated over $
H^{*,*'}(X,\bZ/{\ell})$ by \\$x,\, deg(x)=(1,1)$  and $y, \, deg(y)=(2,1)$\\ with $\beta(x)=y$ and
$x^2=\begin{cases} 0\quad \ell\text{ is odd}\\
                           \tau y+\rho x \quad \ell=2
\end{cases}$\\
where $0\not =\tau\in H^{0,1}(Spec(k),\bZ/\ell)\cong \mu_{\ell}$ and
$\rho=(-1)\in k^{*}/(k^{*})^2\cong K^M_1(k)/2\cong H^{1,1}(Spec(k),\bZ/2).$\\
                           \end{lemma}

 For what follows, we assume that $k$ contains a primitive $\ell^2$-th root of unity
$\xi$, so that $B\bZ/\ell\stackrel{\sim}{\to}B\mu_{\ell}$ and $\beta(\tau)=\xi^{\ell}$ ($=\rho$ for $p=2$) is zero in $k^*/(k^*)^{\ell}=H^{1,1}_{et}(Spec(k);\bZ/{\ell})$. \\

 We will need the following properties:

 \prop\label{propOp}{
 \begin{itemize}
\item [(i)]  $\quad P^i(x)=0\ for\ i>m/2$ and  $x\in H^{m,n}(X,\bZ/{\ell})$;
\item[(ii)] $\quad P^{i}(x)=x^{\ell} \text{ for }  x\in H^{2i,i}(X,\bZ/{\ell})$;
\item [(iii)] for  $X$  smooth  the operation $$Q_i:CH^m(X)/{\ell}=H^{2m,m}(X,\bZ/\ell)\to H^{2m+2\ell^{i}-1,m+(\ell^i-1)}(X,\bZ/{\ell})$$
 is zero ;
\item [(iv)]   $Op.(\tau x)=\tau Op.(x )\quad \text{for} \ Op.=\beta, Q_i \ or\ P^j$;
\item [(v)]   $Q_i=[P^{\ell^{i-1}}, Q_{i-1}].$\\
 \end{itemize}}
 \proof{See \cite[\S 9]{Vo2}; for (iii) one uses that $H^{m,n}(X,\bZ/{\ell})=0$ if $m-2n>0$ and $X$ is a smooth variety over $k$, (iv) follows from the Cartan formula for the motivic cohomology.}

\subsection{Computations for $B\bZ/\ell$}

The computations in this section are similar to \cite{AH62, To97, To99}.

\begin{lemma} In $H^{*,*'}(B\bZ/\ell,\bZ/\ell)$, we have
$Q_i(x)=y^{\ell^i}$ and $Q_i(y)=0$.
\end{lemma}
\begin{proof}  By definition $Q_0(x)=\beta(x)=y$. Using induction and Proposition \ref{propOp}, we compute
\[ Q_i(x)=P^{\ell^{i-1}}Q_{i-1}(x)- Q_{i-1}P^{\ell^{i-1}}(x)=P^{\ell^{i-1}}Q_{i-1}(x)\]
\[  = P^{\ell^{i-1}}(y^{\ell^{i-1}})=y^{\ell^i}.\]
Then $Q_1(y)=-Q_0P^1(y)=-\beta(y^{\ell})=0$. For $i>1$, using induction and Proposition \ref{propOp} again, we conclude that $Q_i(y)=-Q_{i-1}P^{\ell^{i-1}}(y)=0.$\\
\end{proof}

Let $G=(\bZ/\ell)^3$.
 As above, we assume that $k$ contains a primitive $\ell^2$-th root of unity. From Lemma \ref{mot-coh-Bmu}, we have an isomorphism
\[ H^{*,*'}(BG,\bZ/\ell)\cong H^{*,*'}(Spec(k),\bZ/\ell)[y_1,y_2,y_3]\otimes
\Lambda(x_1,x_2,x_3)\]
where $\Lambda(x_1,x_2,x_3)$ is isomorphic to the $\bZ/{\ell}$-module generated by $1$ and $x_{i_1}...x_{i_s}$ for $i_1<...<i_s$
 and $x_ix_j=-x_jx_i\ (i\le j)$, with  $\beta(x_i)=y_i$ and $x_i^2=\tau y_i$ for $\ell=2$. \\
\begin{lemma}\label{op-nonzero-zl}  Let $x=x_1x_2x_3$ in $H^{3,3}(BG,\bZ/\ell)$.  Then
\[Q_iQ_jQ_k(x)\not =0 \in H^{2*,*}(BG,\bZ/\ell)\quad for\ i<j<k.\]
\end{lemma}
\begin{proof} Using Proposition \ref{propOp}(v) and Cartan formula (\ref{propOp}(iv)), we get
\[  Q_k(x)=y_1^{\ell^k}x_2x_3-y_2^{\ell^k}x_1x_3+y_3^{\ell^k}x_1x_2.\]
Then we deduce
\[ Q_iQ_jQ_k(x)=\sum\limits_{\sigma\in S_3} \pm y_{\sigma(1)}^{\ell^k}y_{\sigma(2)}^{\ell^j}y_{\sigma(3)}^{\ell^i}
\not =0\in \bZ/{\ell}[y_1,y_2,y_3].\]
\end{proof}

\section{exceptional Lie groups}\label{Cspaces}

Let $(G,\ell)$ be a simple simply connected Lie group and a prime number from the following list:
\begin{equation}\label{listGl}  (G,\ell)= \begin{cases} G_2, \ell=2,\\
                        F_4, \ell=3,\\
                        E_8,  \ell=5.
                        \end{cases}
\end{equation}

Then $G$ is 2-connected and $H^3(G, \bZ)\cong \bZ$.  Hence
$BG$, viewed as a topological space, is $3$-connected and $H^4(BG, \bZ)\cong \bZ$ (see \cite{Mi-To} for example).
We write  $x_4(G)$ for a generator of $H^4(BG, \bZ)$.

Given a field $k$ with $char(k)\not =\ell$, let us denote by $G_k$
the (split) reductive algebraic group over $k$ corresponding to the Lie group $G$. 

The Chow ring $CH^*(BG_k)$  has been defined  by Totaro  \cite{To99}. More precisely, one has  \begin{equation}\label{BG}BG_k=\varinjlim (U/G_k),\end{equation}
 where $U\subset W$ is an open  set in a linear representation $W$   of $G_k$, such that
 $G_k$ acts freely on $U$. One can then identify $CH^i(BG_k)$
with the group $CH^i(U/G_k)$ if  $\mathrm{codim}_{W}(W\setminus U)>i$,  the group $CH^i(BG_k)$ is then independent of a choice of such $U$. Similarly, one can define
the \'etale cohomology groups $H^{i}_{\acute{e}t}(BG_{k},\bZ_{\ell}(j))$ and the motivic cohomology groups $H^{*,*'}(BG_{ k},\bZ/\ell)$ (see \cite{KN}), the latter  coincide with the motivic cohomology groups of \cite{Mo-Vo} (cf. \cite[Proposition 2.29 and Proposition 3.10]{KN}).  We also have the cycle class map
\begin{equation}\label{cycle-class}\quad cl:CH^*(BG_{\bar k})\otimes\bZ_{\ell}\to
\bigcup_{U}H^{2*}_{\acute{e}t}(BG_{\bar k},\bZ_{\ell}(*))^{U},\end{equation}
where the union is over all open subgroups $U$ of $Gal(\bar k/k).$

The following proposition is known.

\prop\label{cohBGA}{Let $(G,\ell)$ be a group and a prime number from the list (\ref{listGl}). Then \begin{itemize}
\item[(i)] the group $G$ has a maximal elementary non toral subgroup of rank $3$: $$i:A\simeq (\mathbb Z/\ell)^3\subset G;$$
\item[(ii)] $H^4(BG,\mathbb Z/\ell)\simeq \mathbb Z/\ell$, generated by the image $x_4$ of the generator $x_4(G)$ of $H^4(BG,\mathbb Z)\simeq \mathbb Z$; 
\item[(iii)]   $Q_1(i^*x_4)=Q_1Q_0(x_1x_2x_3)$, in the notations of Lemma \ref{op-nonzero-zl}. In particular,  $Q_1(i^*x_4)$ is non zero.
\end{itemize}
}
\proof{ For $(i)$ see \cite{Gr}, for the computation of the cohomology groups with $\mathbb Z/\ell$-coefficients in $(ii)$ see \cite{Mi-To} VII 5.12;  $(iii)$ follows from \cite{Ka-Te-Ya} for $\ell=2$ and \cite[Proposition 3.2]{Ka-Ya} for $\ell=3,5$ (see \cite{Ka-No} as well). \qed}

\section{Algebraic approximation of $BG$}\label{Proj}

Write
\begin{equation}\label{defBG}
BG_k=\varinjlim(U/G_k)
\end{equation}
as in the previous section.
Using proposition \ref{cohBGA} and a specialization argument, we will first
 construct a quasi-projective algebraic variety $X$ over $k$ as a quotient $X=U/G_k$ (where  $codim_{W}(W\setminus U)$ is big enough), such that the cycle class map (\ref{cnt}) is not surjective for such $X$.
However,  if one is interested only in quasi-projective counterexamples for the surjectivity of the map (\ref{cnt}), one can produce more naive
examples, for instance as a complement of some smooth hypersurfaces in a projective space. Hence we are interested to find an approximation of Chow groups and the  \'etale cohomology of $BG_{\bar k}$ as a smooth and projective variety. In the case when the group $G$ is finite, this is done in \cite[Th\'eor\`eme 2.1]{CTSz10}.  In this section we give such an approximation for the groups we consider here, this construction is suggested by B. Totaro.

\begin{Propsd}\label{AlgAppr}  {\it Let $G$ be a compact Lie group as in (\ref{listGl}).
For all but finitely many primes $p$ there exists a smooth and projective variety $X_{k}$ over a finite field $k$ with  $char\,k=p$, an element $x_{4,k}\in H^4_{\acute{e}t}(B(\mathbb G_m\times G_{\bar k}), \mathbb Z_{\ell}(2))$, invariant under the action of  $Gal(\bar k/k)$ and a map
$\tau: X_{k}\to B(\bG_m\times G_k)$ in the category $\mathcal H_{\cdot}(k)$ such that 
\begin{itemize}
\item[(i)] $y_{4,k}=\tau^*pr_2^* x_{4,k}$ is a non zero class in $H^4_{\acute{e}t}(X_{\bar k}, \mathbb Z_{\ell}(2))/torsion$, where $pr_2:\mathbb G_m\times G_k\to G_k$ is the projection on the second factor;
 \item[(ii)] the operation $Q_1(\bar  y_{4,k})$ is non zero,  where we  write $\bar y_{4,k}$ for the image of $ y_{4,k}$ in $H^4_{\acute{e}t}(X_{\bar k}, \mathbb Z/\ell).$
 \end{itemize}}
\end{Propsd}

\rem{For the purpose of this note, the proposition above is enough. See also \cite{Ek} for a a general statement on the projective approximation of the cohomology of classifying spaces. \\}

 Theorem \ref{ceProj} now  follows from the proposition above:\\
 
{\it Proof of theorem \ref{ceProj}.\\}
For $k$ a finite field and $X_{k}$  as in the proposition above, we find a nontrivial class $y_{4,k}$ in its cohomology in degree $4$ modulo torsion, which is not annihilated by the operation $Q_1$. This class can not be algebraic by proposition \ref{propOp}(iii). \qed\\

{\it Proof of proposition \ref{AlgAppr}.\\}
We proceed in three steps. First, we construct a quasi-projective approximation in a family parametrized by $Spec\,\mathbb Z$. Then, for the geometric generic fibre we produce a projective approximation, by a topological argument. We finish the proof by specialization. \\

{\it Step 1: construction of a family.\\}
Let $\mathcal G$ be a split reductive group over $B=Spec\,\mathbb Z$ corresponding to $G$, such a group exists by \cite{SGA3} XXV 1.3.
As $B$ is an affine scheme of dimension $1$, we can embed $\mathcal G$ as a closed subgroup of $GL_{d,B}$ for some $d$ (see \cite{SGA3} $\mathrm{VI}_B$ 13.2 and 13.5).  Moreover,  one can assume that $\mathcal G\hookrightarrow PGL_{d, B}$ such that this embedding lifts to $\mathcal H=GL_{d, B}$, up to remplacing $B$ by an open subset (e.g. using the map
$ A\mapsto \begin{pmatrix}1&0&0\\
0&-1&0\\
0&0&A
\end{pmatrix}$ and changing $d$ by $d+2$).

By a construction of \cite[Remark 1.4]{To99} and \cite[Lemme 9.2]{CTSa07},  there exists $n>0$, a linear $\mathcal H$-representation $\mathcal {O}_B^{\oplus n}$  and an $\mathcal H$-invariant open subset $\mathcal U\subset \mathcal {O}_B^{\oplus n}$, which one can assume flat over $B$, such that the action of $\mathcal H$ is free on $\mathcal U$.
Let $\mathcal V_N=\mathcal {O}_B^{\oplus Nn}$.  Then the group $PGL_{n, B}$ acts on $\bP(\mathcal V_N)$ and,
 taking $N$ sufficiently large, one can assume that the action is free
outside a subset $S$ of high codimension  $s\geq 4$.

By  restriction, the group $\mathcal G$ acts on  $\bP(\mathcal V_N)$ as well,
let $\mathcal Y=\bP(\mathcal V_N)//\mathcal G$ be the GIT quotient for this action \cite{Mu65, Sec}. The scheme $\mathcal Y$ is projective over $B$ and we fix an embedding $\mathcal Y \subset \bP^M_B$.
Let \begin{equation}\label{deff}f:\mathcal W\to B\end{equation} be the open set of $\mathcal Y$
corresponding to the quotient of the open set  $\mathcal U$ as above where $\mathcal G$ acts
freely. From the construction,  $\mathcal Y- \mathcal W$ has high codimension in $\mathcal Y$.

For any point $b\in B$ with residue field $\kappa(b)$,
the fibre $\mathcal W_b$ is a smooth quasi-projective variety and if $N$ is big enough, we have isomorphisms by lifting  $\mathcal G$ to $GL_{n, B}$
(cf. p. 263 in \cite{To99})
$$\mathcal W_b\cong (\bP(\mathcal V_N)-S)_b/\mathcal G_b \cong ((\mathcal V_N-\{0\})/\bG_m-S)_b)/\mathcal G_b
\cong (\mathcal V_N-S')_b/(\bG_m
\times \mathcal G)_b$$
where $S'=pr^{-1}S\cup\{0\}$ for  the projection $pr:(\mathcal V_N-\{0\})\to \bP(\mathcal V_N)$.
Hence we have isomorphisms
\begin{equation}\label{isob}H^{i}(\mathcal W_b,\bZ_{\ell})\stackrel{\sim}{\to}  H^{i}(B(\bG_m\times \mathcal G)_b,\bZ_{\ell}) \text{ for }i\leq s, \ell\neq char\, {\kappa(b)},
\end{equation}
 induced by a natural map $\mathcal W_b\to B(\bG_m\times \mathcal G_b)$ from the presentation (\ref{defBG}).\\

{\it Step 2: the generic fibre.\\}
Let $Y=\mathcal Y_\mathbb C$ and $W=\mathcal W_\mathbb C$ be the geometric generic fibres of $\mathcal Y$ and $\mathcal W$ over $B$.
Consider a general linear space $L$ in $\bP^M$ of codimension equal to
$1+dim(Y-W)$. Then $L\cap Y=L\cap W$ so $X:=L\cap W$ is a smooth
projective  variety. Note that one can assume that $L$ is defined over $\mathbb Q$.

By a version of the Lefschetz hyperplane theorem for
quasi-projective varieties, established by Hamm (as a special case of  Theorem II.1.2 in \cite{Go-Ma}),
for  $V\subset \bP^M$ a closed complex subvariety of dimension $d$,
not necessarily smooth,
$Z\subset V$ a closed subset, and $H$ a hyperplane in $\bP^M$,
if $V-(Z\cup H)$ is local complete intersection (e.g. $V-Z$ is smooth)
then
\[\pi_i((V-Z)\cap H)\to \pi_i(V-Z)\]
is an isomorphism for $i<d-1$ and surjective for $i=d-1$. In particular, $H^i((V-Z)\cap H, \bZ)\to H^i(V-Z, \bZ)$
is an isomorphism for $i<d-1$ and surjective for $i=d-1$
by the Whitehead theorem.  %(Hurewicz theorem).

We then deduce that \begin{equation}\label{isoC}H^{i}(X,R)\stackrel{\sim}{\to}  H^{i}(B(\bG_m\times G),R) \text{ for }i\leq s \text{ and } R=\bZ \text{ or } \bZ/n.     
                    \end{equation}
                    Hence $H^{i}_{\acute{e}t}(X,\bZ/n )\stackrel{\sim}{\to}  H^{i}_{\acute{e}t}(B(\bG_m\times G),\bZ/n), i\leq s$. Note that as the cohomology of $BG$ is a direct factor in the cohomology of $B(\mathbb G_m\times G)$, we get that  $x_4(G)$ (with the notations of the previous section) generates a direct factor isomorphic to $\mathbb Z_{\ell}$ in the cohomology group $H^{4}_{\acute{e}t}(X,\bZ_{\ell})$.\\

{\it Step 3: specialization argument.\\}
We can now specialize the construction above to obtain the statement over a finite field. 

More precisely, one can find a dense open set $B'\subset B$ and a linear space $\mathcal L\subset \bP_{B'}^M$ such that $\mathcal L_{\mathbb C}\simeq L$ and such that for any $b\in B'$ the fibre $\mathcal X_b$ of $\mathcal X=\mathcal L\cap \mathcal Y$ is smooth.  Up to passing to an \'etale cover of $B'$,  one can assume that the inclusion $(\mathbb Z/\ell)^3\subset G_{\mathbb C}$ from proposition \ref{cohBGA} extends an inclusion $i:\mathcal A=(\mathbb Z/{\ell})^3_{B'}\hookrightarrow \mathcal G_{B'}$ (cf. \cite{SGA3} XI.5.8). 

Let $b\in B'$ and let $k=\kappa(b)$. As the schemes $\mathcal X$, $\mathcal Y$ and $\mathcal U/\mathcal A$ are smooth over $B'$, we have the following commutative diagram, where the vertical maps are induced by the specialisation maps: 

{\footnotesize
 $$\xymatrix{
H^{4}_{\acute{e}t}(X,\bZ_{\ell}(2))\ar[d]& H^{4}_{\acute{e}t}(Y,\bZ_{\ell}(2))\ar[l]\ar[d]\ar[r] & H^{4}_{\acute{e}t}(\mathcal U_{\mathbb C}/(\mathbb Z/\ell)^3,\bZ/{\ell})\ar[d]  &   H^{4}_{\acute{e}t}(B(\mathbb Z/\ell)^3,\bZ/{\ell})\ar@{=}[d]\ar_{\simeq}[l]  &\\
H^{4}_{\acute{e}t}(\mathcal X_{\bar k},\bZ_{\ell}(2))& H^{4}_{\acute{e}t}(\mathcal Y_{\bar k},\bZ_{\ell}(2))\ar[l]\ar[r] & H^{4}_{\acute{e}t}(\mathcal U_{\bar k}/(\mathbb Z/\ell)^3,\bZ/{\ell})&  H^{4}_{\acute{e}t}(B(\mathbb Z/\ell)^3,\bZ/{\ell})\ar_{\simeq}[l] &
}
$$
}

The left vertical map is an isomorphism since $\mathcal X$ is proper. Hence we get a class $y_{4,k}\in H^{4}_{\acute{e}t}(\mathcal X_{\bar k},\bZ_{\ell}(2)),$ corresponding to $x_4(G)\in H^{4}_{\acute{e}t}(X,\bZ_{\ell}(2))$. The map $H^{4}_{\acute{e}t}(Y,\bZ_{\ell}(2))\to H^{4}_{\acute{e}t}(X,\bZ_{\ell}(2))$ is an isomorphism by step $2$, so that $y_{4,k}$ comes from an element $x_{4,k}\in H^{4}_{\acute{e}t}(\mathcal Y_{\bar k},\bZ_{\ell}(2)).$ Let $z_{4,k}\in H^{4}_{\acute{e}t}(B(\mathbb Z/\ell)^3,\bZ/{\ell})$ be the image of $x_{4,k}$.  From the diagram and proposition \ref{cohBGA} we deduce that   $Q_1(z_{4,k})=Q_1Q_0(x_1x_2x_3)\neq 0$, hence $Q_1(\bar y_{4,k})$ is non zero as well. From the construction, the class $y_{4,k}$ generates a subgroup of $ H^{4}_{\acute{e}t}(\mathcal X_{\bar k},\bZ_{\ell}(2)),$  which is a direct factor isomorphic to $\mathbb Z_{\ell}$, and is Galois-invariant. Letting $X_k=\mathcal X_{k}$ this finishes the proof of the proposition.

\qed

\rem{ We can also adapt the arguments of \cite[Th\'eor\`eme 2.1]{CTSz10} to  produce projective examples with higher torsion non-algebraic classes, while in {\it loc.cit.} one constructs   $\ell$-torsion classes.
Let $G(n)$ be the finite group $G(\bF_{\ell^{n}})$, so that we have
$$\varprojlim H^*_{\acute{e}t}(BG(n),\bZ_{\ell})=
H^*_{\acute{e}t}(BG_{\bar k},\bZ_{\ell}).$$
Then,
following the construction in {\it loc.cit.} one gets
\begin{quote}{\it
 For any $n>0$, there exits a positive integer $i_n$ and
 a Godeaux-Serre variety  $X_{n, \bar k}$  for the finite
group $G(n)$ such that
\begin{itemize}
\item[(1)]\quad  $x\in H_{\acute{e}t}^4(X_{n,\bar k};\bZ_{\ell}(2))$ generates
$\bZ/{\ell^{n'}}$ for some $n'\ge n$;
\item [(2)]\quad $x$ is not in the image of the cycle class map (\ref{iTate}).\\
\end{itemize}}
\end{quote}


\begin{thebibliography}{Ra-Wi-Ya}

  \bibitem{AH62} M. F. Atiyah, F. Hirzebruch, {\it Analytic cycles on complex manifolds}, Topology {\bf 1} (1962), 25 -- 45.

\bibitem{CTSa07} J.-L. Colliot-Th\'el\`ene et J.-J. Sansuc, {\it The rationality problem for fields of invariants under linear algebraic groups (with special regards to the Brauer group)}, Algebraic groups and homogeneous spaces, 113--186, Tata Inst. Fund. Res. Stud. Math., {\bf 19}, Tata Inst. Fund. Res., Mumbai, 2007.


  \bibitem{CTSz10} J.-L. Colliot-Th\'el\`ene et T. Szamuely, {\it Autour de la conjecture de Tate \`a coefficients $\Z_{\ell}$ sur les corps finis}, The Geometry of Algebraic Cycles (ed.  Akhtar, Brosnan, Joshua), AMS/Clay Institute Proceedings (2010), 83--98.

    \bibitem{Go-Ma} M. Goresky and R. MacPherson, {\it
Stratified Morse theory},
Ergebnisse der Mathematik und ihrer Grenzgebiete (3) [Results in Mathematics and Related Areas (3)], {\bf 14}. Springer-Verlag, Berlin, 1988.

\bibitem{Gr} R. L. Griess, {\it Elementary abelian $p$-subgroups of algebraic groups},
Geom. Dedicata {\bf 39} (1991), no. 3, 253--305. 

\bibitem{Ek} T. Ekedahl, {\it Approximating classifying spaces by smooth projective varieties},  	arXiv:0905.1538.

\bibitem{KN} B. Kahn et T.-K.-Ngan Nguyen, {\it Modules de cycles et classes non ramifi\'ees sur un espace classifiant},  arXiv:1211.0304.



\bibitem{Ka-Ya} M. Kameko and N. Yagita, {\it Chern subrings}, Proc. Amer. Math. Soc. {\bf 138} (2010), no.~1, 367--373.

\bibitem{Ka-No} R. Kane and D. Nothbohn, {\it Elementary abelian $p$-subgroups of Lie groups},
Publ. Res. Inst. Math. Sci. {\bf 27} (1991), no. 5, 801--811. 


\bibitem{Ka-Te-Ya} M. Kameko, M. Tezuka and N. Yagita, {\it Coniveau spectral sequences of classifying spaces for exceptional and Spin groups}, Math. Proc. Cambridge Phil. Soc.
\textbf{98} (2012), 251--278.



\bibitem{Ko90} J. Koll\'ar, {\it In
Trento examples},  in
{\it Classification of irregular
varieties}, edited by E. Ballico, F. Catanese, C. Ciliberto, Lecture Notes in
Math. {\bf 1515}, Springer (1990).


\bibitem{Mi07} J. S. Milne, {\it The Tate conjecture over finite fields (AIM talk)}, 2007.

\bibitem{Mi-To} M. Mimura and H. Toda, {\it Topology of Lie groups}, I and II, Translations of Math. Monographs, Amer. Math. Soc, {\bf 91} (1991).

\bibitem{Mo-Vo}
F.Morel and V.Voevodsky, \emph{ $\mathbb A^1$-homotopy theory of schemes}, Publ.Math. IHES,
\textbf{90} (1999), 45--143.

\bibitem{Mu65} D. Mumford,  {\it Geometric invariant theory},  Ergebnisse der Mathematik und ihrer Grenzgebiete, Neue Folge, Band {\bf 34} Springer-Verlag, Berlin-New York 1965.

\bibitem{Ri12} J. Riou, {\it Op\'erations de Steenrod motiviques}, preprint, 2012.

\bibitem{Sch98} Ch. Schoen, {\it An integral analog of the Tate conjecture for one-dimensional cycles on varieties over finite fields}, Math. Ann. {\bf 311} (1998), no. 3, 493-500.

\bibitem{Sec} C.S. Seshadri, {\it Geometric reductivity over arbitrary base}, Adv. Math.,
{\bf 26} (1977)  225--274.

\bibitem{Ta65} J. Tate, {\it Algebraic cycles and poles of zeta functions}, Arithmetical algebraic geometry (Proc. Conf. Purdue Univ. 1963), 93 -- 110, Harper and Row, New York (1965).


\bibitem{To97} B. Totaro, {\it Torsion algebraic cycles and complex cobordism}, J. Amer. Math. Soc. {\bf 10} (1997), no. 2, 467--493.


\bibitem{To99} B. Totaro, {\it The Chow ring of a classifying space}, in  "Algebraic K-theory", ed. W. Raskind and C. Weibel, Proceedings of Symposia in Pure Mathematics,  \textbf{67}, American Mathematical Society (1999), 249--281.


\bibitem{Vo2}
V.Voevodsky, {\it Reduced power operations in motivic cohomology}, Publ. Math. IHES
\textbf{98} (2003),1--57.


\bibitem[SGA3]{SGA3} M. Demazure et A. Grothendieck, \emph{Sch\'emas en groupes},  S\'eminaire de G\'eom\'etrie Alg\'ebrique du Bois Marie SGA 3,  Lecture Notes in Math. \textbf{151, 152, 153}, Springer, Berlin-Heildelberg-New York, 1977, r\'e\'edition Tomes I, III, Publications de la SMF, Documents math\'ematiques {\bf 7, 8} (2011).




\end{thebibliography}
\end{document}